\documentclass[a4paper]{article}
\usepackage{amssymb,amsmath,amsthm,hyperref,enumitem,tikz,csquotes,subcaption,cleveref,wrapfig,url,authblk,doi}
\usepackage[margin=1in]{geometry}
\usetikzlibrary{matrix}
\tikzset{font=\footnotesize} 
\makeatletter
\newcommand\bibstyle@comma{\bibpunct(),a,,}
\makeatother

\newtheorem{theorem}{Theorem}[section]
\newtheorem{lemma}[theorem]{Lemma}
\newtheorem{corollary}[theorem]{Corollary}
\theoremstyle{definition}
\newtheorem{remark}[theorem]{Remark}
\newtheorem{example}[theorem]{Example}

\newcommand{\ex}{\mathbb{E}}
\newcommand{\m}{\mathrm{m}}
\newcommand{\h}{\mathrm{h}}
\newcommand{\e}{\mathrm{\ell}}

\newcommand{\g}{\mathrm{g}}
\renewcommand{\d}{\mathrm{d}}
\newcommand{\du}{\mathop{\d u}}
\newcommand{\F}{\bar{F}}
\renewcommand{\(}{\left(}
\renewcommand{\)}{\right)}
\newcommand{\mrl}{\preceq_{\text{mrl}}}
\newcommand{\hr}{\preceq_{\text{hr}}}
\newcommand{\st}{\preceq_{\text{st}}}

\begin{document}

\title{A Class of Distributions for Linear Demand Markets}
\author[1]{Stefanos Leonardos}
\author[2]{Costis Melolidakis}
%\author[kaq]{Georgios Piliouras}\ead{georgios@sutd.edu.sg}
\affil[1]{Singapore University of Technology and Design, 8 Somapah Rd, 487372 Singapore, {stefanos\_leonardos@sutd.edu.sg}}
\affil[2]{National and Kapodistrian University of Athens, Panepistimioupolis, 15784 Athens, Greece, {cmelol@math.uoa.gr}}

\maketitle

\begin{abstract}
In this paper, we study distributions that describe markets with linear stochastic demand. We express the price elasticity of expected demand in terms of the \emph{mean residual demand} (MRD) function of the demand distribution and characterize optimal prices or equivalently, points of unitary elasticity, as fixed points of the MRD function. This leads to economic interpretable conditions on the demand distribution under which such fixed points exists and are unique. In particular, markets with increasing price elasticity of expected demand that eventually become elastic correspond to distributions with \emph{decreasing generalized mean residual demand} (DGMRD) and finite second moment. DGMRD distributions strictly generalize the widely used \emph{increasing generalized failure rate} (IGFR) distributions. In real life economic applications, they arise naturally as mixtures of (possibly) IGFR distributions over disjoint intervals. We further elaborate on the relationship of the two classes and link their limiting behavior at infinity. We examine moment and closure properties of the DGMRD distributions that are important in economic applications and illustrate our results with examples.
\end{abstract}


\section{Introduction}\label{intro}

\subsection{Problem Formulation and Motivation}
Optimal pricing of monopolistic services and goods under uncertain demand is a recurrent theme in the revenue management literature. A non-exhaustive list includes \cite{Ra81}, \cite{Ra95}, \cite{Da01}, \cite{Zi06,Co12} and more recently \cite{Co15,Ch17}. The tractability of this problem is closely related to the unimodality of the associated revenue function or equivalently to the existence of a unique optimal price for the seller. Accordingly, a central question in this line of research is the study of conditions on the distribution of the source of uncertainty that yield a unimodal revenue function. The particular case in which the monopolist is selling a single good and uncertainty concerns the valuation of the buyer has been settled in \cite{La99,La01} and \cite{Be07}. Specifically, if the distribution of the valuation satisfies the \emph{increasing generalized failure} (IGFR) property, then the seller's revenue function is unimodal. The class of distributions with the IGFR property includes most distributions that are commonly used in economic applications \cite{Pa05,La06,Ba13}.\par
In the present paper, we are concerned with the study of unimodality conditions in a more general formulation of this problem. Specifically, we consider a seller who is selling physical goods to a buyer that may buy several units. The buyer is privately informed about her type $\alpha$, while the seller only knows the distribution of $\alpha$. Here, $\alpha$ can be interpreted as the \emph{demand level} or more roughly as the amount of goods that the buyer is willing to buy. Equivalently, $\alpha$ can be thought of as the market demand in the presence of several buyers, each of which is willing to buy one unit of the service or good. In any case, the seller's expected revenue function is given by 
\begin{equation}\label{revenue}R\(p\)=p\ex \(D\(p\mid \alpha\)\)\end{equation}
where $p$ denotes the seller's price, $D\(p\mid \alpha\)$ the demand at price $p$, given that the buyer's realized type is $\alpha$ and $\ex$ the expectation over the distribution of $\alpha$. We assume that $D\(p\mid \alpha\)$ is continuous and non-increasing in $p$. The seller's objective is to determine the optimal price $p^*$ that maximizes $R\(p\)$. By differentiating $R\(p\)$, the seller's first order condition can be written as 
\begin{equation}\label{fixedpoint} p=-\frac{\ex\(D\(p\mid \alpha\)\)}{\frac{\d}{\d p}\ex\(D\(p\mid \alpha\)\)}\end{equation}
Given that $\varepsilon\(p\):=- \frac{\d \ex\(D\(p\mid \alpha\)\)/ \ex\(D\(p\mid \alpha\)\)}{\d p/ p}$ is the price elasticity of expected demand, cf. \cite{Xu10}, the solutions of \eqref{fixedpoint} correspond to the points of unitary price elasticity of expected demand. \par
Depending on the specific expression for $D\(p\mid \alpha\)$, \eqref{fixedpoint} may have a single, multiple or even no solutions. However, in this setting, the IGFR condition may not directly apply to yield a unimodality condition since the expression in \Cref{fixedpoint} requires information about the whole range of the distribution (evaluation of the conditional expected demand and its derivative) and not only about its local behavior at the current demand level. In addition, the IGFR condition -- although particularly inclusive in terms of common distributions \cite{Ba13} -- is restricted to distributions that are defined over connected intervals \cite{La06}. This poses a restriction to study scenario analysis of real-life economic applications, at which sellers often weigh different beliefs over disjoint intervals that correspond to low, modal and high (extreme) demand realizations. 

\subsection{Model and Results} 
Motivated by the shortcomings of the IGFR property to simplify the seller's pricing problem in this setting, we seek to formulate an alternative condition on the distribution of the random demand -- or equivalently on the seller's belief about it -- that will yield a unimodal revenue function in equation \eqref{revenue}, i.e., a unique solution to equation \eqref{fixedpoint}. Our focus is on the particular instantation of the additive demand model introduced by \cite{Mi59}, with the common assumption of linear deterministic component, studied (among others) in \cite{Pe99, Hu13} and \cite{Co15}. Specifically, let $D\(p\mid \alpha\)=\(\alpha-p\)_+$, where $\alpha$ denotes the random demand level. We assume that $\alpha$ is a non-negative random variable with continuous cummulative distribution function (cdf) $F$, tail $\F:=1-F$ and finite expectation\footnote{For one of our results, \Cref{thm:main}, we will also require that $\ex \alpha^2$ is also finite. However, unless stated otherwise, we do not make this assumption.}, $\ex \alpha<+\infty$. For the support of $\alpha$, let $L:=\sup{\{p\ge0, F\(p\)=0\}}\ge0$ and $H:=\inf{\{p\ge0: F\(p\)=1\}}\le +\infty$. Using this notation, \eqref{fixedpoint} can be expressed in terms of the \emph{mean residual demand} (MRD) function, defined as
\begin{equation}\label{mrl}\m\(p\):=\begin{cases}\ex\(\alpha-p\mid \alpha>p\)=\displaystyle\frac{1}{\F\(p\)}\int_{p}^{+\infty}\F\(u\)\du, & \text{if } p<H\\0, & \text{otherwise}\end{cases}
\end{equation}
In reliability theory, $\m\(p\)$ is known as the mean residual life (MRL) function, see, e.g., \cite{Sh07} or \cite{Lax06}. In particular, solutions of \eqref{fixedpoint} are precisely solutions of the fixed point equation $p=\m\(p\)$. This is shown in \Cref{lem:easy}. Such fixed points may also be of interest in problems of broader economic and mathematical context, see e.g., \cite{Ha81} and \cite{Ba05}. \par
According to the previous discussion, our aim is to study fixed points of the MRD function, i.e., solutions to the equation $\m\(p\)=p$ for $p>0$. To study this equation, we introduce the \emph{generalized mean residual demand} (GMRD) function, $\e\(p\):=\m\(p\)/p$, for $0<p<H$, cf. \eqref{gmrl}, which corresponds to the inverse of the price elasticity of expected demand. It follows that prices $p^*$ with unitary price elasticity which maximize the seller's expected revenue, satisfy $\e\(p^*\)=1$ or equivalently $p^*=\m\(p^*\)$. In turn, this implies that a sufficient condition for the unimodality of the seller's expected revenue function in a market with linear demand is that the MRD function of the associated demand distribution has a unique fixed point. If the expected demand has increasing price elasticity and eventually becomes elastic then such a fixed point exists and is unique. In terms of the demand distribution this is equivalent to the property that $\e\(p\)$ is decreasing and eventually becomes less than $1$. The derivation of necessary and sufficient conditions for the unimodality of the expected revenue function is the main result of \Cref{motivation} and is formally established in \Cref{thm:main}.\par
An immediate implication of this result is that markets with increasing price elasticity of expected demand can be modelled via distributions that satisfy the \emph{decreasing generalized mean residual demand} (DGMRD) property. As mentioned above, if demand uncertainty is exogenous and corresponds to the buyer's valuation for a single product unit, increasingly elastic markets are described by distributions with \emph{increasing generalized failure rate} (IGFR), see \cite{La01} and \cite{Be07}. In \Cref{specialcase}, we formulate this as a subcase of the current problem. To further elaborate on their connection to the present setting, we analyze the relationship of IGFR and DGMRD distributions and study their properties. In \Cref{classes}, we provide an alternative proof to the well known fact (see \cite{Be98,Ka14} that DGMRD distributions generalize the IGFR distributions and establish that the converse is also true if the MRD function is log-convex. A commonly used distribution that is DGMRD but not IGFR is the Birnbaum-Saunders distribution for specific values of its parameters, cf. \Cref{birnbaum}. \par
DGMRD distributions arise naturally in economic applications as mixtures of (possibly IGFR) distributions. In particular, they provide the possibility to study distributions defined over disjoint, continuous intervals and thus, provide a useful generalization over IGFR distributions for practical pricing scenarios\footnote{Discrete IGFR distributions offer another possibility to study such cases \cite{Ba13}.}. Since IGFR distributions are restricted to continuous intervals, this property can be exploited to deal with demand distributions that are concentrated around several distinct levels, such as low and high or low, high and intermediate demand. This allows the study of multiple scenarios in the same model to account for the probability of demand shocks triggered by unpredictable catastrophic events, technological innovations or abrupt changes in consumers' brand preferences \cite{Ma03,Lo09,Ba17}. Specifically, a seller may encounter a demand distribution that is concentrated with high probability over a connected interval -- main scenario -- and with lower probabilities over extreme, but smaller intervals, that correspond to less likely, but very high or very low demand realizations. In general, dropping the requirement of a connected interval allows for greater flexibility to the theoretical modelling of situations that arise in common business practice \cite{Ya17}. From a technical perspective, the current analysis retains only the minimum requirement that $F$ is continuous. This is satisfied as long as the distribution of the random demand is atomless, i.e., as long as there do not exist single points with positive probability, even if the distribution is supported over disjoint intervals\footnote{In technical terms, this means that they analysis does not require $F$ to be absolutely continuous, i.e., to possess a density $f=F'$. In fact, the analyis extends even to singular distributions but since such examples are not relevant for economic applications, we defer their discussion in a different context, see \cite{Leo21}.}. \par
We proceed with the study of moment and closure properties of DGMRD distributions that are useful in economic modelling. In \Cref{moments}, we show that the moments of DGMRD distributions with unbounded support are linked to their limiting behavior at infinity. Specifically, if the GMRD function tends to $c\ge0$ as $p\to+\infty$, then for any $n>0$, its $\(n+1\)$-th moment is finite if and only if $c<1/n$. This implies, that markets with increasing and eventually elastic demand, i.e., $\e\(p\)<1$ for every $p$ sufficiently large, correspond to DGMRD distributions with finite second moment. Hence, \Cref{moments} allows the formulation of a technical condition as a moment condition. In \cite{Leo20}, we show that in the equilibrium analysis of horizontal competition, the number $n$ of competitors imposes a finiteness condition on the $n$-th moment of the distribution. In \Cref{limiting}, we study the relationship in the limiting behavior of the GMRD and GFR functions and link \Cref{moments} to Theorem 2 of \cite{La06}. In sum, \Cref{classes,limiting} along with \Cref{birnbaum,ex:uniform} establish the relationship and highlight the differences between the DGMRD and IGFR classes of distributions. \par
Finally, we examine closure properties of the DGMRD and of the smaller \emph{decreasing MRD} (DMRD) class of distributions, and compare our findings with \cite{Pa05} and \cite{Ba13}. Such properties are relevant for the modelling of economic applications in which the potential seller updates her information about the demand distribution, aggregates different demands, reeestimates her expectations about the demand or gains access to more concrete information. In mathematical terms, these actions are expressed via increasing or decreasing transformations, convolutions and shifting, cf. \Cref{thm:concave} and \Cref{cor:shift} or scale transformations and truncations, cf. \Cref{thm:truncate}. We conclude our presentation with a discussion of the limitations of the current results along wih open questions in \Cref{sec:discussion}.

\subsection{Related Literature}\label{sub:related}
The MRD and GMRD functions have been studied in \cite{Ha81} and \cite{Gu88} and more recently in the survey of \cite{Ba05} in the context of reliability and statistical analysis with scarce references to economic applications. In revenue management, the MRD and GMRD functions naturally arise in problems with demand uncertainty. \cite{Ma18,Lu16,So09,So08,Pe99}, and references cited therein, study the tail of the distribution of the source of uncertainty, see e.g., \cite{So09}, Lemma 1 and \cite{So08}, equation $\(2\)$. In all these cases, the use of the currently defined GMRD function aids for a more succinct representation and the DGMRD condition provides a potential technical condition for refinement of the respective results. Similar (unimodality and elasticity) conditions are studied in \cite{Be05,Ko11,Lu13} and, in a spirit more similar to ours, in \cite{La06,Be07} and \cite{Ba13}. Questions that deal with demand uncertainty are formulated in \cite{Co15} and \cite{Ch17}. Their findings shed new light on the exact trade-off between generality of technical assumptions on the demand distribution and limitations on the demand curve. In particular, both papers provide novel perspectives on the microfoundations of the linear demand model -- e.g., as a good approximation of various demand curves -- and, thus, justify its use in a wider (than previously thought) spectrum of real-life applications.
\par
In \cite{Leon21}, we study the problem of monopoly pricing under demand uncertainty in a vertical market. First, we recover the present expression for the price elasticity of expected demand and the same unimodality conditions in a relatively more general setting. Unlike the current technical treatment, we then turn to the effects on optimal pricing of the various market characteristics, as expressed by features of the demand distribution. They key insight is that the present characterization of the seller's optimal price via the MRD function, allows for the application of the theory of stochastic orderings, \cite{Sh07,Lax06}. This leads to a diverse comparative statics analysis on the various demand features that challenges existing insights about the effects of market size, demand transformations and demand variability on monopoly pricing. Under various perspectives, demand uncertainty in supply chains (verticl markets) is also studied by \cite{Wo97,Li05,Xu10} and \cite{Li17} for more general distributions (i.e., beyond the linear model), but typically, under the more restrictive IFR assumption on the demand distribution.

\section{Unimodality of the seller's revenue function: the DGMRD property}\label{motivation} 
Our goal in this Section is to establish necessary and sufficient conditions for the unimodality of the seller's revenue function. This is the statement of \Cref{thm:main}, which crucially relies on the expression of the price elasticity of expected demand via the \emph{generalized mean residual demand (GMRD)} function which is derived next. To that end, we first express \eqref{fixedpoint} in terms of the MRD function $\m\(p\)$. Under mild analytical assumptions on $D\(p\mid \alpha\)$, we have that $\frac{\d}{\d p}\ex\(D\(p\mid \alpha\)\)=\ex\(\frac{\partial}{\partial p} D\(p\mid \alpha\)\)$, \cite{Fl73}. However, in the specific case that $D\(p\mid \alpha\)=\(\alpha-p\)_+$, this can be derived in a straightforward way as shown in \Cref{diffa}.
\begin{lemma}\label{diffa}
If $\alpha$ is a non-negative random variable with finite expectation $\ex\alpha<+\infty$ and continuous distribution function $F$, then $\frac{\d}{\d p} \ex\(\alpha-p\)_+=\ex \(\frac{\partial \(\alpha-p\)_+}{\partial p}\)=-\F\(p\)$ for any $p>0$.
\end{lemma}
\begin{proof}[\textbf{Proof}] Let $K_{h}\(\alpha\):=-\frac1h\,\left[\(\alpha-p-h\)_+-\(\alpha-p\)_+\right]$ and take $h>0$. Then, $K_h\(\alpha\)= \mathbf{1}_{\{\alpha>p+h\}} + \dfrac{\alpha-p}{h}\mathbf{1}_{\{p< \alpha \le p+h\}}$ and therefore $\lim_{h\to 0+}K_{h}\(\alpha\)=\mathbf{1}_{\{\alpha>p\}}$. Since $0\le K_h\(\alpha\) \le 1$ for all $\alpha$, the dominated convergence theorem implies that $\lim_{h\to0+}\ex\(K_h\(\alpha\)\) = P\(\alpha>p\)$. In a similar fashion, one may show that $\lim_{h\to0-}\ex\(K_h\(\alpha\)\) = P\(\alpha \geq p\)$. Since the distribution of $\alpha$ is non-atomic, $P\(\alpha>p\) = P\(\alpha \geq p\)$ and hence, $\lim_{h\to0}\ex\(K_h\(\alpha\)\) = \F\(p\)$. By the definition of $K_h\(\alpha\)$, it follows that $\lim_{h\to0}\ex\(K_h\(\alpha\)\) = -\frac{\d}{\d p}\ex\(\alpha-p\)_+$, which concludes the proof.
\end{proof}
\noindent Using \Cref{diffa}, the following formulation of \eqref{fixedpoint} is now immediate.
\begin{lemma}\label{lem:easy}
In the linear demand case $D\(p\mid \alpha\)=\(\alpha-p\)_+$, the seller's first order condition, \eqref{fixedpoint}, can be written as 
\begin{equation}\label{fixedmrl}p=\m\(p\)\end{equation} where $\m\(p\)$ denotes the MRD function of the demand distribution. \end{lemma}
\begin{proof}
Since $\(\alpha-p\)_+$ is non-negative, we may write \[\ex\(\alpha-p\)_{+} = \int_{0}^{\infty}P\(\(\alpha-p\)_{+}>u\) \du =\int_{p}^{\infty}\F\(u\)\du,\] for $0 \le p<H$, see \cite{Bi86}. Using \eqref{mrl}, we thus, have $\ex\(\alpha-p\)_+=\m\(p\)\F\(p\)$ and \eqref{fixedpoint} takes the form $p=\m\(p\)$.
\end{proof}

From the buyer's revenue maximization perspective, we are interested in conditions for the existence and uniqueness of solutions of \eqref{fixedmrl}. To study this problem, we define the \emph{generalized mean residual demand (GMRD)} function
\begin{equation}\label{gmrl}\e\(p\):=\frac{\m\(p\)}{p}=\frac{1}{p\F\(p\)}\int_{p}^{+\infty}\F\(u\)\du\end{equation}
for all $0<p<H$. We say that a random variable $D$ has the \emph{DGMRD property}, if $\e\(p\)$ is non-increasing in $p$ for $0<p<H$. While the MRD function at a point $p$ expresses the expected additional demand given that current demand has reached (or exceeded) the threshold $p$, the GMRD function expresses the corresponding expected additional demand as a percentage of the current demand. From an economic perspective, $\e\(p\)$ has an appealing interpretation, since it is the inverse of the price elasticity of the \emph{expected} demand, $\varepsilon\(p\):=-p\cdot\frac{\d}{\d p}\ex\(D\(p\mid\alpha\)\)/\ex\(D\(p\mid\alpha\)\)$,
\begin{equation}\label{elasticity}\e\(p\)=\frac{\m\(p\)}{p}=\(\frac{\F\(p\)}{\m\(p\)\F\(p\)}\cdot p\)^{-1}=\varepsilon\(p\)^{-1}\end{equation}
Thus, demand distributions with the DGMRD property precisely capture markets of goods with increasing price elasticity of expected demand. Moreover, together with \eqref{fixedmrl}, \eqref{elasticity} implies that the seller's revenue is maximized at prices $p^*$ with unitary price elasticity of \emph{expected} demand. In non-trivial, realistic problems, demand eventually becomes elastic, see also \cite{La06}. Accordingly, let $p_1:=\sup{\{p\ge0:\e\(p\)\ge 1\}}$ and assume that $p_1<+\infty$ or equivalently that the price elasticity of expected demand, eventually becomes greater than $1$. For a continuous distribution $F$ with finite expectation $\ex \alpha$, such that $F\(0\)=0$, we have that $\m\(0\)=\ex \alpha>0$ and hence, $p_1>0$. Combining the above, we obtain necessary and sufficient conditions for the unimodality of the seller's revenue function $R\(p\)$, or equivalently for the existence and uniqueness of a solution of \eqref{fixedmrl}.
\begin{theorem}\label{thm:main} Suppose that $\alpha$ is a random variable with continuous distribution $F$, $F\(0\)=0$, and finite expectation, such that $p_1<+\infty$. The seller's revenue function $R\(p\)=p\ex\(\alpha-p\)_+$ is maximized at all points $p^*$ with unitary elasticity of expected demand, i.e., at all points $p^*$ that satisfy $\e\(p\)=1$ or equivalently, $p^*=\m\(p^*\)$. If $\e\(p\)$ is strictly decreasing, then a fixed point $p^*$ exists and is unique. 
\end{theorem}
\begin{proof}
To establish the first part, it remains to check that any point satisfying \eqref{fixedmrl} corresponds to a maximum under the assumption that $\e\(p\)$ is strictly decreasing. Clearly, $\e\(p\)$ is continuous and since $\m\(0\)=\ex\alpha<+\infty$, we have that $\lim_{p\to 0^+}\e\(p\)=+\infty$. Hence, for values of $p$ close to $0$, demand is inelastic and the seller's revenue increases as prices increase. However, the limiting behavior of $\e\(p\)$ as $p$ approaches $H$ from the left may vary, depending on whether $H$ is finite or not. If $H$ is finite, i.e., if the support of $\alpha$ is bounded, then $\lim_{p\to H-}\e\(p\)=0$.  Hence, in this case, demand eventually becomes elastic and a critical point $p^*\in\(0,H\)$ that maximizes $R\(p\)$ exists without any further conditions. The assumption that $\e\(p\)$ is strictly decreasing, establishes the uniqueness of $p^*$. If $H=+\infty$, then an optimal solution $p^*$ may not exist because the limiting behavior of $\m\(p\)$, as $p\to+\infty$, may vary, see e.g., the Pareto distribution in \Cref{pareto}. However, under the assumption that $\e\(p\)$ is strictly decreasing and that $p_1<+\infty$, such a critical $p^*$ exists and is unique.
\end{proof}

\begin{remark}
The assumption $p_1<+\infty$ is equivalent to the condition that the distribution of $\alpha$ has finite second moment. Indeed, as we show in \Cref{moments}, if the support of $\alpha$ is unbounded, and $\e\(p\)$ is decreasing, then, $\lim_{p\to+\infty}\e\(p\)<1$ if and only if $\ex \alpha^2$ is finite. The assumption of strict monotonocity eliminates intervals with $\m\(p\)=p$, in which multiple consecutive solutions occur. However, it may be relaxed to weak monotonicity without significant loss of generality. This relies on the explicit characterization of distributions with MRD functions that contain linear segments which is given in Proposition 10 of \cite{Ha81}. Namely, $\m\(p\)=p$ on some interval $J=[a,b]\subseteq[L,H]$ if and only if $\F\(p\)p^2=\F\(a\)a^2$ for all $p\in J$. If $J$ is unbounded, this implies that $\alpha$ has the Pareto distribution on $J$ with shape parameter $2$. In this case, $\ex \alpha^2=+\infty$, see \Cref{pareto}, which is precluded by the requirement that $p_1<+\infty$. Hence, to replace strict by weak monotonicity, it suffices to exclude distributions that contain intervals $J=[a,b]\subseteq[L,H]$ with $b<+\infty$ in their support, for which $\F\(p\)p^2=\F\(a\)a^2$ for all $p\in J$.
\end{remark}

\subsection{Special Case: Uncertain Reservation Price with Single Product Unit}\label{specialcase}
The case in which uncertainty corresponds to the buyer's valuation, see \cite{La06} and \cite{Zi06}, can be derived as a subcase of \eqref{fixedpoint}. In particular, assume that the seller posts a price $p$ and the buyer's reservation price is $\alpha$ which is randomly drawn from a distribution $F$. If $\alpha\ge p$, then the buyer buys one unit of the product, otherwise she does not buy. This implies that $D\(p\mid \alpha\)=\mathbf{1}{\{p\le \alpha\}}$ and hence, that $\ex\(D\(p\mid \alpha\)\)=\F\(p\)$. In this case and under the assumption that $F$ is absolutely continuous, with $F'=f$, \eqref{fixedpoint} takes the form $p\h\(p\)=1$, for $p<H$, where $\h\(p\):=f\(p\)/\F\(p\)$ is the hazard rate function of $\alpha$. \cite{La99,La01} and \cite{Be07} define $g\(p\):=p\h\(p\)$ as the \emph{generalized failure rate (GFR)} function of $\alpha$ and show that if $\alpha$ has the \emph{increasing generalized failure rate} (IGFR) property, i.e., if $\g\(p\)$ is non-decreasing in $p$ for $p<H$, and if $\g\(p\)$ eventually exceeds $1$, then the seller's optimal price exists and is unique. The GFR function, $\g\(p\)$, corresponds to the price elasticity of demand and hence the assumptions that $\g\(p\)$ is increasing and eventually exceeds $1$ capture the economic intuition of increasing and eventually elastic demand. Similarly to \Cref{thm:main}, the optimal seller's price $p^*$ coincides with the point of unitary price elasticity, i.e., $\g\(p^*\)=1$. \par
The GFR function was introduced in economic applications by \cite{Si76}, who used it to model income distributions. It was further studied in the same context by \cite{Be95} and \cite{Be98} who provide an alternative definition of the IGFR property without requiring the existence of a density. In the context of revenue management, properties of IGFR distributions have been studied by \cite{Zi04}, \cite{Pa05}, \cite{La06} and \cite{Ba13}.

\section{Properties of DGMRD Distributions}\label{secDGMRD}
\Cref{motivation} motivates the study of DGMRD distributions as class of distributions that arise naturally in a seller's pricing optimization problem when the seller is facing a linear stochastic demand. It turns out, that the class of DGMRD distributions is general enough to include as a subclass the IFR, DMRD and IGFR distributions that are widely used in revenue management applications. This statement along with several analytical and closure properties of the DGMRD distributions are established next.\par
For the remaining part, let $X\sim F$ be a non-negative random variable, with support in $L, H$ as in \Cref{intro}, continuous distributions function $F$, tail $\F:=1-F$ and finite expectation $\ex X<+\infty$. Let $\m\(x\)$ denote the MRD function of $X$, as defined in \eqref{mrl}, and $\e\(x\)$ denote the GMRD function of $X$, as defined in \eqref{gmrl}. We say that distribution $X$ has the \emph{decreasing MRD} (DMRD) property, or simply that $X$ is DMRD, if $\m\(p\)$ is non-increasing in $p$ for $p<H$. If additionally, $X$ is an absolutely continuous random variable with $F'=f$ almost everywhere, for some density function $f$, one can easily verify that the derivative $\m'\(p\)$ exists and is given by 
\begin{equation}\label{derivative}\m'\(p\)=\h\(p\)\m\(p\)-1\end{equation} 
where $\h\(p\)=f\(p\)/\F\(p\)$ denotes the \emph{hazard rate function} of $X$, see e.g., \cite{Br03}. 

\subsection{The DGMRD and IGFR Classes of Distributions}\label{sub:classes}
To compare the IGFR and DGMRD classes, we restrict attention to non-negative, absolutely continuous random variables. We, then have
\begin{theorem}\label{classes}
If $X$ is a non-negative, absolutely continuous random variable, with $\ex X<+\infty$, then
\begin{enumerate}[label=$\(\roman*\)\;$, leftmargin=0cm,itemindent=.5cm,labelwidth=\itemindent,labelsep=0cm, align=left, noitemsep,nolistsep]
\item If $X$ is IGFR, then $X$ is DGMRD.
\item If $X$ is DGMRD and $\m\(x\)$ is $\log$-convex, then $X$ is IGFR. 
\end{enumerate}
\end{theorem}
Part $\(i\)$ of \Cref{moments}, has been already been stated by \cite{Be98} and \cite{Ka14}. To derive an alternative proof of part $\(i\)$ and to establish part $\(ii\)$ of \Cref{classes}, we will use the notions of stochastic orderings, see \cite{Sh07} or \cite{Be16}. Let $X_i$ be random variables with distribution, failure rate and MRD functions denoted by $F_i,\h_i$ and $\m_i$ respectively, for $i=1,2$. $X_1$ is said to be smaller than $X_2$ in the \emph{usual stochastic order}, denoted by $X_1\st X_2$, if $F_2\(x\)\le F_1\(x\)$ for all $x\in \mathbb R$. Similarly, $X_1$ is said to be smaller than $X_2$ in the \emph{failure or hazard rate order}, denoted by $X_1\hr X_2$, if $h_2\(x\)\le h_1\(x\)$ for all $x\in\mathbb R$. Finally, $X_1$ is said to be smaller than $X_2$ in the \emph{mean residual life order}, denoted by $X_1\mrl X_2$, if $\m_1\(x\)\le\m_2\(x\)$ for all $x\in\mathbb R$. 
\begin{proof}[Proof of \Cref{classes}]
By Theorem 1 of \cite{La06}, $X$ is IGFR if and only if $X\hr\lambda X$ for all $\lambda\ge1$. By Theorem 2.A.1 of \cite{Sh07}, if $X\hr\lambda X$, then $X\mrl\lambda X$. Now, $\m_{\lambda X}\(x\)=\lambda\cdot\m\(x/\lambda\)$. Hence, for $\lambda\ge1$, $X\mrl\lambda X$ is by definition equivalent to $\m\(x\)\le \m_{\lambda X}\(x\)$ for all $x>0$, which in turn is equivalent to $\e\(x\)\le \e\(x/\lambda\)$ for all $x>0$. As this holds for any $\lambda\ge 1$, the last inequality is equivalent to $\e\(x\)$ being decreasing, i.e., to $X$ being DGMRD.\par
To prove the second part of the Theorem, it suffices to show that $\m\(x\)/\m_{\lambda X}\(x\)$ is increasing in $x$, for $0<x<H$ and all $\lambda\ge 1$. Indeed, if this is the case, Theorem 2.A.2 of \cite{Sh07} implies that $X\mrl \lambda X$ for all $\lambda \ge 1$ is equivalent to $X\hr\lambda X$ for all $\lambda \ge 1$, which as we have seen, is equivalent to $X$ being IGFR. Since $\m_{\lambda X}\(x\)=\lambda\m\(x/\lambda\)$ and $\m\(x\)$ is differentiable, $\m\(x\)/\m_{\lambda X}\(x\)$ is increasing in $x\in\(0,H\)$ for all $\lambda\ge1$ if and only if $\frac{\d}{\d x}\(\frac{\m\(x\)}{\lambda\m\(x/\lambda\)}\)\ge 0$, for all $\lambda \ge 1$, i.e., if and only if $\frac{\m'\(x\)}{\m\(x\)}\ge \frac{\m'\(x/\lambda\)}{\lambda\m\(x/\lambda\)}$, for all $\lambda \ge 1$. This is equivalent to $\frac{\d}{\d x}\ln\(\m\(x\)\)$ being increasing in $x$, i.e., to $\m\(x\)$ being $\log$-convex.
\end{proof}

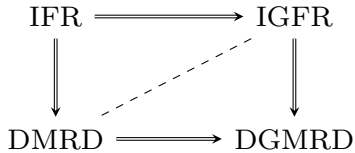
\begin{figure}[!htb]
\centering\vspace{-0.15cm}
\begin{tikzpicture}[scale=4/3, every node/.style={scale=4/3}]
  \matrix (m) [matrix of math nodes,row sep=3em,column sep=4em,minimum width=2em]
  { \text{IFR} &\text{IGFR} \\
    \text{DMRD} & \text{DGMRD} \\};
  \path[-stealth] (m-1-1) edge [double] (m-2-1) edge [double] (m-1-2) 
    (m-1-2) edge [double] (m-2-2)
    (m-2-1.east|-m-2-2) edge [double] (m-2-2)
    (m-1-2) edge [dashed,-] (m-2-1);
\end{tikzpicture}
\caption{Relationship between the IFR, IGFR, DMRD and DGMRD classes of distribution. The IFR property implies both the IGFR and the DMRD properties, which in turn both imply the DGMRD property. All inclusions are proper. The DMRD property neither implies nor is implied by the IGFR property (in the general case), however the IGFR property seems more inclusive in terms of families of distributions that are useful in pricing applications. Finally,  \Cref{classes}-(ii) provides a (rather restrictive) condition under which a DGMRD distribution is also IGFR.}
\label{fig:relations}
\end{figure}
%\begin{wrapfigure}[9]{R}{0.3\textwidth}
%\centering\vspace{-0.15cm}
%\begin{tikzpicture}
%  \matrix (m) [matrix of math nodes,row sep=3em,column sep=4em,minimum width=2em]
%  { \text{IFR} &\text{IGFR} \\
%    \text{DMRD} & \text{DGMRD} \\};
%  \path[-stealth] (m-1-1) edge [double] (m-2-1) edge [double] (m-1-2) 
%    (m-1-2) edge [double] (m-2-2)
%    (m-2-1.east|-m-2-2) edge [double] (m-2-2)
%    (m-1-2) edge [dashed,-] (m-2-1);
%\end{tikzpicture}
%\caption{Relationship between the IFR, IGFR, DMRD and DGMRD classes.}
%\label{fig:relations}
%\end{wrapfigure}
Based on the proof of \Cref{classes}, a DGMRD random variable $X$ is not IGFR if there exists $\lambda \ge 1$ such that $X$ is smaller than $\lambda X$ in the mean residual life order but not in the hazard rate order. Although more involved, the present derivation of part $\(i\)$ utilizes the characterization of both IGFR and DGMRD in terms of stochastic orderings -- $\hr$ for IGFR and $\mrl$ for DGMRD -- and thus, points to the sufficiency condition of part $\(ii\)$. Specifically, in view of the proof of part $\(i\)$, the proof of part $\(ii\)$ reduces to finding conditions, under which, the mean residual life order implies the hazard rate order. Such conditions are provided in Theorem 2.A.2 of \cite{Sh07}. However, as \cite{Sh07} already point out, the condition of $\log$-convexity is restrictive and indeed there are many distributions with $\log$-concave MRD function that are nevertheless IGFR. Hence, it would be of interest to obtain part $\(ii\)$ of \Cref{classes} under a more general condition. \par
Conceptually, the GFR and GMRD functions differ in the same sense that the FR and MRD functions do. Namely, while the GFR function at a point $x$ provides information about the instantaneous behavior of the distribution just after point $x$, the GMRD function provides information about the entire expected behavior of the distribution after point $x$. As the IGFR is trivially implied by the IFR property, the same holds for the DGMRD and DMRD properties. The relationships between all four classes are shown in \Cref{fig:relations}. The IGFR property does not imply, nor is implied by the DMRD property. However, the former seems more inclusive than the latter, cf. \cite{Ba05}, Table 3 and \cite{Ba13}, Table 1. Conversely, DMRD distributions that are not IGFR can be constructed by considering random variables without a connected support. This relies on the observation that if a distribution $X$ is IGFR, then its support must be an interval, see \cite{La01}. However, it remains unclear whether or not, the DMRD property implies the IGFR property when attention is restricted to absolutely continuous random variables with connected support. A commonly used distribution that is DGMRD but not IGFR is the Birnbaum-Saunders distribution.

\begin{example}[Birnbaum-Saunders distribution]\label{birnbaum} The Birnbaum-Saunders (BS) distribution, which is extensively used in reliability applications, see \cite{Jo95}, provides an example of a random variable which is DGMRD but not IGFR for certain values of its parameters. The pdf of $X$ is 
\[f\(x\)=\frac{1}{2a x \sqrt{2\pi}}\(\sqrt{\frac{x}{\beta}}+\sqrt{\frac{\beta}{x}}\)\exp{\(-\frac{1}{2a^2}\(\sqrt{\frac{x}{\beta}}-\sqrt{\frac{\beta}{x}}\)^2\)}, \qquad \text{ for } x>0,\]
where $a>0$ is the shape parameter and $\beta>0$ is the scale parameter. In particular, let $X\sim \text{BS}$ with parameters $a=6$ and $\beta=5$. Using the formula for $f\(x\)$, \Cref{fig:birnbaum} can be obtained numerically. 
\begin{figure}[!htp]
\centering
\includegraphics[width=\textwidth]{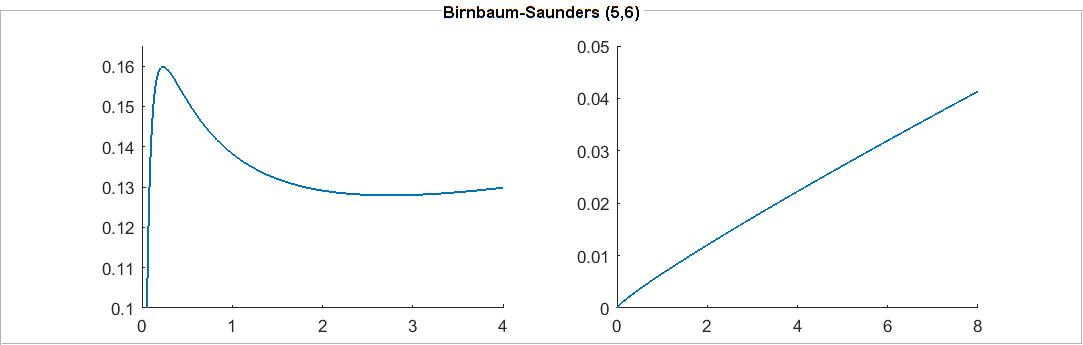}
\caption{Birnbaum-Saunders distribution for $a=6, \beta=5$. The GFR function (left panel) is not monotone increasing in contrast to the price elasticity of expected demand (right panel) which is the inverse of the GMRD function.}
\label{fig:birnbaum}
\end{figure}
Implementing the BS distribution for different $\beta$ and $\gamma$, shows that, unlike other distribution families, as e.g., the Gamma or Beta, the shapes of the GFR and GMRD functions of the BS distribution depend largely on the exact values of its parameters. For different values of its parameters, the BS distribution has either increasing or bathtub-shaped (first decreasing and then increasing) MRD function, \cite{Ta99}.
\end{example}

\subsection{Mixtures of DGMRD Distributions over Disjoint Intervals}\label{sub:mixtures}
As mentioned above, IGFR random variables must have a connected support. Under certain circumstances, this property poses restrictive limitations in economic modelling. For instance, when a seller is uncertain about the exact support of the demand, their belief can be naturally expressed as a mixture of two or more distributions over disjoint intervals. In this case, even if each individual distribution is IGFR, their mixture is certainly not. In this respect, the DGMRD property is more promising, since mixtures of IGFR distributions may still be DGMRD. However, in general, different mixtures of IGFR distributions may or may not be DGMRD even if the only difference is in the mixing weights. Such a case is illustrated in \Cref{ex:uniform}. 
\begin{example}[Mixture of Uniform distributions on disjoint intervals]\label{ex:uniform}
 Let $U\(L,H\)$ denotes the uniform distribution on $\(L,H\)$ and let $X_1\sim U\(1,2\)$ with cdf $F_1$ and $X_2\sim U\(3,4\)$ with cdf $F_2$. Further, let $X_{\lambda}$ with cdf $F_{\lambda}=\lambda F_1+\(1-\lambda\) F_2$ for $\lambda \in \(0,1\)$ describe the seller's belief about the demand. Both $X_1,X_2$ are IFR, hence IGFR, DMRD and DGMRD.
The support of $X_{\lambda}$ is not connected, hence $X_{\lambda}$ is not IGFR for $0<\lambda<1$. Contrarily, the GMRD $\e_{\lambda}$ of $X_\lambda$ is given by
\[\e_{\lambda}\(x\)=\begin{cases}\hfill \lambda \e_1\(x\)+\(1-\lambda\)\e_2\(x\), & 0<x\le 1 \\[0.1cm]\dfrac{\lambda\(2-x\)\e_1\(x\)+\(1-\lambda\)\e_2\(x\)}{\lambda\(2-x\)+\(1-\lambda\)}, & 1\le x\le 2\\ \hfill \e_2\(x\), & 2\le x<4\end{cases}\]
Hence, $\e_\lambda\(x\)$ is decreasing for $x\notin [1,2]$. For $x\in [1,2]$, a direct substitution shows that $\e_{1/4}\(x\)$ is decreasing over $[1,2]$, hence $X_{1/4}$ is DGMRD, while $\e_{3/4}\(x\)$ is first decreasing and then increasing, as shown in \Cref{mixture} and hence $X_{3/4}$ is not DGMRD.
\end{example}
Due to the importance of mixtures of distributions in economic applications, the derivation of conditions under which mixtures of IGFR (or DGMRD) distributions remain DGMRD is an interesting open question. Formally, this question can be formulated as follows. Consider a finite set of probability density functions (pdfs) $f_1\(x\),\dots, f_n\(x\)$ and/or corresponding cumulative distribution functions (cdfs) $F_1\(x\),\dots, F_n\(x\)$ and weights $w_1,\dots, w_n$ such that $w_i \ge 0$ for $i=1,\dots,n$ and $\sum_{i=1}^n w_i = 1$. Then, the pdf $f\(x\)$ and the cdf $F\(x\)$ of the \emph{mixture distribution} can be represented as a convex combination of the individual pdfs or cdfs, respectively, as follows
\begin{align*}
f\(x\):=\sum_{i=1}^nw_if_i\(x\), \qquad F\(x\):=\sum_{i=1}^nw_iF_i\(x\)
\end{align*}
A particular instantiation of the above that covers a broad range of economic applications is to consider three cdfs with disjoint support that correspond to low, modal and high demand realizations, respectively. Using the notation $l,m,h$ to denote the respective weights, with $l+m+h=1$, the resulting demand distribution takes the form $F=lF_1+mF_2+hF_3$. \Cref{ex:uniform} demonstrates that even in the simplest possible case of two different mixtures of two IFR distributions, the resulting distribution may or may not be DGMRD, even if the only difference between the two mixtures is in the weights. Accordingly, the derivation of necessary and/or sufficient conditions under which such mixtures retain the DGMRD property, i.e., the derivation of closure properties under mixtures of the DGMRD class of distributions, remains an interesting open question. In a related study that may prove useful in this direction, \cite{Nav08} confirm that mixtures of standard IFR (and hence DGMRD) distributions, e.g., exponential, may not be DGMRD (it may be bathtub-shaped), and derive sufficient conditions under which asymptotical monotonicity is retained. 
 
\subsection{Limiting Behavior \& Moments of DGMRD Distributions}
The moments of DGMRD distributions with unbounded support are closely linked with the limiting behavior of the GMRD function $\e\(x\)$, as $x\to+\infty$.
\begin{theorem}\label{moments}
Let $X$ be a non-negative DGMRD random variable with $\ex X<+\infty$ and $\displaystyle\lim_{x\to+\infty}\e\(x\)=c$. If $\beta>0$, then
 $c<\frac1\beta$, if and only if $\ex X^{\beta+1}<+\infty$. In particular, $c=0$ if and only if $\ex X^{\beta+1}<+\infty$ for every $\beta>0$.
\end{theorem}
\begin{wrapfigure}[15]{r}{0.49\textwidth}
\centering\vspace{-0.1cm}
\includegraphics[width=0.45\textwidth]{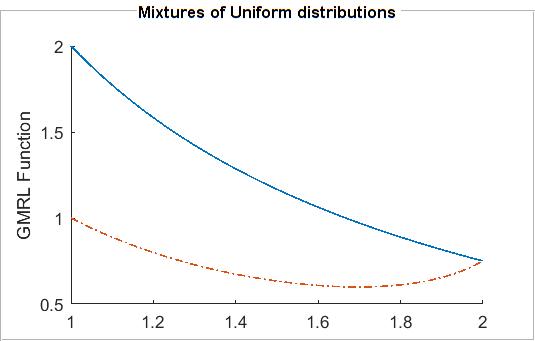}
\caption{\label{mixture}The GMRD function of $X_\lambda$ for $\lambda=1/4$ (solid) and $\lambda=3/4$ (dotted).}
\end{wrapfigure} 
For the proof of \Cref{moments}, we utilize the theory of regularly varying distributions, see \cite{Fe71, Ha81} and \cite{Gu13}. First, observe that if $X$ is a non-negative random variable, then by a simple change of variable, one may rewrite\footnote{By differentiating this expression, provided that $F'=f$ almost everywhere, one obtains an alternative straightforward proof that IGFR implies DGMRD.} $\e\(x\)$ in \eqref{gmrl} as $\e\(x\)=\int_{1}^{+\infty}\frac{\F\(ux\)}{\F\(x\)}\du$. Since we have assumed that $\ex X <+\infty$, $\e\(x\)$ is well defined. We say that $\F$ is \emph{regularly varying at infinity} with exponent $\rho \in \mathbb R$, if $\F\(ux\)/\F\(x\)\to u^{\rho}$ for all $u\ge0$ as $x\to+\infty$. In this case, we write $\F\in \mathcal {RV}\(\rho\)$. If $\F\(ux\)/\F\(x\)\to \infty$ for $0<u<1$ and $\F\(ux\)/\F\(x\)\to 0$ for $u>1$ as $x\to+\infty$, then we say that $\F$ is \emph{rapidly varying at infinity with exponent $-\infty$} or simply that $\F$ is \emph{rapidly varying}, in symbols $\F \in \mathcal{RV}\(-\infty\)$. If $\F \in \mathcal{RV}\(\rho\)$ with $\rho\in \mathbb R$, then we can write $\F$ as $\F\(u\)=u^{\rho}Z\(u\)$, where $Z$ is regularly varying at infinity with exponent $\rho=0$. In this case, we say that $Z$ is \emph{slowly varying at infinity} and write $Z \in \mathcal{SV}$. \cite{Fe71}, see Section VIII.8, shows that if $Z\(u\)>0$ and $Z\in \mathcal{SV}$, then the integral $\int_{0}^{+\infty}u^\rho Z\(u\)\du$ is convergent for $\rho<-1$ and divergent for $\rho>-1$. We are now ready to prove \Cref{moments}.

\begin{proof}[Proof of \Cref{moments}.] Let $c>0$. Then, the convergence of $\ell\(x\)$ to some $c \in \(0,+\infty\)$ is equivalent to $\F$ being regularly varying at infinity with exponent $-1-\frac1c$, in symbols $\F \in \mathcal{RV}\(-1-\frac1c\)$, see Proposition 11(b) of \cite{Ha81}. Hence, there exists a function $Z \in \mathcal{SV}$, such that $\F\(x\)=x^{-1-\frac1c}Z\(x\)$. Since $X$ is non-negative, this implies that for any $\beta>0$, we may write $\ex X^{\beta+1}=\int_{0}^{+\infty}\(\beta+1\)u^{\beta}\F\(u\)\du =\(\beta+1\)\int_{0}^{+\infty}u^{\beta-1-\frac1c}Z\(u\)\du$. Using \cite{Fe71}, the latter integral converges for $\beta<\frac1c$ and diverges for $\beta>\frac1c$. For $c=\frac1\beta$, we employ the approach of \cite{La06} and compare $X$ with a random variable $Y\sim \text{Pareto}\(1,\beta+1\)$, where $1$ is the location parameter and $\beta+1$ the shape parameter. In this case $\m_Y\(x\)=x/\beta$ and $\ex Y^{\beta+1}=+\infty$, which may be used to conclude that $\ex X^{\beta+1}=+\infty$ as well. To see this, observe that since $\e\(x\)$ is decreasing to $1/\beta$ by assumption, we have that $\m_X\(x\)\ge x/\beta=\m_{Y}\(x\)$ and hence $Y\mrl X$. Moreover, $\frac{\m_Y\(x\)}{\m_X\(x\)}=\frac{1}{\beta}\cdot\frac{1}{\e\(x\)}$, which by assumption increases in $x$ for all $x>0$. This implies that $Y$ is smaller than $X$ in the hazard rate order, see Theorem 2.A.2 of \cite{Sh07}, and hence also in the usual stochastic order, i.e., $Y\st X$. Hence, $\ex X^{\beta+1}\ge\ex Y^{\beta+1}=+\infty$. \par 
If $c=0$, then $\F\(x\)$ is rapidly varying with exponent $-\infty$, i.e., $\F \in \mathcal{RV}\(-\infty\)$, see Proposition 11(c) of \cite{Ha81}. It is known, see \cite{Ha70}, that all moments of rapidly varying distributions are finite. Conversely, if $\ex X^{\beta+1}<+\infty$ for every $\beta>0$, then it is a straightforward implication that $c=0$.\end{proof}
\Cref{moments} can be compared with Theorem 2 of \cite{La06}, who states an analogous result for IGFR distributions. \Cref{limiting} establishes the link between the two.
\begin{theorem}\label{limiting}
Let $X$ be an absolutely continuous, non-negative random variable with unbounded support and $\ex X<+\infty$. If $\lim_{x\to+\infty}\g\(x\)$ exists and is equal to $\kappa$ with $\kappa>1$ (possibly infinite), then
\begin{equation}\label{reciprocal}\lim_{x\to+\infty}\e\(x\)=1/\(\kappa-1\)\end{equation}
\end{theorem}
\begin{proof}
Since $\ex X<+\infty$, both $\lim_{x\to+\infty}\int_{x}^{+\infty}\F\(u\)\du$ and $\lim_{x\to+\infty}x\F\(x\)$ are equal to $0$. To compute $\lim_{x\to+\infty}\ell\(x\)$ we use L'H{\^o}pital's rule. We have that $\frac{\d}{\d x}\int_{x}^{+\infty}\F\(u\)\du=-\F\(x\)$ and $\frac{\d}{\d x}\(x\F\(x\)\)=\F\(x\)\(1-\g\(x\)\)$. Hence, under the assumption that $\lim_{x\to+\infty}\g\(x\)=\kappa$, we conclude that \[\lim_{x\to+\infty}\ell\(x\)=\lim_{x\to+\infty} \frac{1}{\g\(x\)-1}=\frac{1}{\kappa-1}.\]
\end{proof}

The inverse relationship in the limiting behavior of $\e\(x\)$ and $\g\(x\)$ in \eqref{reciprocal} is similar in flavor to equation $\(2.1\)$ of \cite{Br03}. In the case that $\kappa<+\infty$, Theorem 2 of \cite{La06} restricted to $n>1$, follows from \Cref{classes,moments}, and equation \eqref{reciprocal}. This approach also covers the case $n=\kappa$, which is not considered in the proof by \cite{La06}. As for IGFR distributions, the Pareto distribution provides a limiting case between decreasing and increasing GMRD distributions, since it is the unique distribution with constant GMRD function.

\begin{example}[Pareto distribution]\label{pareto} Let $X$ be Pareto distributed with pdf $f\(x\)=kL^kx^{-\(k+1\)}\mathbf{1}_{\{L\le x\}}$, and parameters $L>0$ and $k > 1$ (for $0<k\le 1$ we get $\ex X=+\infty$, which contradicts the basic assumptions of our model). To simplify, let $L=1$, so that $f\(x\)=k x^{-k-1}\mathbf{1}_{\{1 \leq x \}}$, $\F\(x\) = x^{-k}\mathbf{1}_{\{1 \leq x\}}$, and $\ex X=\frac{k}{k-1}$. The mean residual demand of $X$ is given by $\m\(x\)=\frac{x}{k-1}+\frac{k}{k-1}\(1-x\)_+$ and, hence, is decreasing for $x<1$ and increasing for $x\ge 1$. However, the GMRD function $\e\(x\)=\frac{1}{k-1}$ is decreasing for $0<x<1$ and constant for $x\ge1$, hence, $X$ is DGMRD. Similarly, for $1\le x$ the failure (hazard) rate $\h\(x\)=kx^{-1}$ is decreasing, but the generalized failure rate $\g\(x\)=k$ is constant and, hence, $X$ is IGFR. In this case, the seller's payoff function, \eqref{revenue}, becomes 
\begin{equation*}R\(x\)=x\m\(x\)\F\(x\)=\begin{cases}x\(\dfrac{k}{k-1}-x\), & \text{if } 0 \le x < 1\\[0.2cm]
\dfrac{ x^{2-k}}{\(k-1\)}, & \text{if } x \ge 1, \end{cases} \end{equation*}
which diverges as $x\to +\infty$, for $k < 2$ and remains constant for $k=2$. In particular, for $k\le 2$, the second moment of $X$ is infinite, i.e., $\ex X^2=+\infty$, and also $\lim_{x\to+\infty}\e\(x\)=\frac{1}{k-1}\ge1$ and $\lim_{x\to+\infty}\g\(x\)=k\le2$, which agrees with \Cref{moments}. 
On the other hand, for $k > 2$, there exists a unique fixed point $x^* = \frac{k}{2\(k-1\)}$, as expected. 
\end{example}

\section{Closure Properties of the DGMRD Class of Distributions}\label{subdmrl}
\cite{Pa05} and \cite{Ba13} study closure properties of the IFR and IGFR classes under operations that involve continuous transformations, truncations, and convolutions. Such operations are important in economic applications, as they can be used to model changes or updates in the seller's beliefs (transformations and truncations) or aggregation of demands from different markets (convolutions). Resembling the IFR when compared to the IGFR class, the DMRD class exhibits better closure properties than the DGMRD class.
\begin{theorem}\label{thm:concave}
Let $X$ be a non-negative, absolutely continuous, DMRD random variable and let $\phi:\mathbb R_+\to \mathbb R_+$ be a strictly increasing, concave and differentiable function. Then, $Y:=\phi\(X\)$ is DMRD. 
\end{theorem}
\begin{proof}
Let $F$ denote the cdf of $X$, $f$ its pdf and $\h$ its hazard rate. Then, for $y>0$, $F_Y\(y\)=F\(\phi^{-1}\(y\)\)$ and $f_Y\(y\)=f\(\phi^{-1}\(y\)\)\frac{1}{\phi'\(\phi^{-1}\(y\)\)}$, where $\phi^{-1}$ denotes the inverse of $\phi$. Hence $\m_Y\(y\)=\(\F\(\phi^{-1}\(y\)\)\)^{-1}\cdot\int_{y}^{+\infty}\F\(\phi^{-1}\(u\)\)\du=\(\F\(\phi^{-1}\(y\)\)\)^{-1}\cdot\int_{\phi^{-1}\(y\)}^{+\infty}\F\(u\)\phi'\(u\)\du$. By \eqref{derivative}, and since $\h_Y\(y\)=\h\(\phi^{-1}\(y\)\)\cdot \frac{1}{\phi'\(\phi^{-1}\(y\)\)}$, we conclude that $\m'_Y\(y\)=\h\(\phi^{-1}\(y\)\)\cdot \(\F\(\phi^{-1}\(y\)\)\)^{-1}\cdot\int_{\phi^{-1}\(y\)}^{+\infty}\F\(u\)\frac{\phi'\(u\)}{\phi'\(\phi^{-1}\(y\)\)}\du-1$. Concavity of $\phi$ implies that for $u>\phi^{-1}\(y\)$, $\frac{\phi'\(u\)}{\phi'\(\phi^{-1}\(y\)\)}\le 1$. Thus, $\m'_Y\(y\)\le \h\(\phi^{-1}\(y\)\)\m\(\phi^{-1}\(y\)\)-1=\m'\(\phi^{-1}\(y\)\)\le 0$, since $\m\(y\)$ is decreasing by assumption.
\end{proof}
Hence, the class of absolutely continuous, DMRD random variables is closed under strictly increasing, differentiable and concave transformations. 
By \Cref{thm:concave}, it is immediate that
\begin{corollary}\label{cor:shift}
Let $X$ be a non-negative, absolutely continuous, DMRD random variable. Then,
\begin{enumerate}[label=$\(\roman*\)\;$,leftmargin=0cm,itemindent=.5cm,labelwidth=\itemindent,labelsep=0cm, align=left, noitemsep, nolistsep]
\item for any $\alpha>0$ and $\beta \in \mathbb R$, $\alpha X+\beta$ is DMRD, (i.e., the DMRD class is closed under positive scale transformations and shifting).
\item for any $0<\alpha\le 1$, $X^{\alpha}$ is DMRD.  
\end{enumerate}
\end{corollary}
More results about the DMRD class can be found in \cite{Ba05, Lax06} and \cite{Sh07}. Turning to the DGMRD class, it is straightforward (thus omitted) to show that the DGMRD property is preserved under positive scale transformations and left truncations. For a random variable $X$ with support inbetween $L$ and $H$, and any $\alpha\in \(L,H\)$, the left truncated random variable $X_{\alpha}$ is defined as $X_{\alpha}=X\mathbf{1}_{\{X\ge \alpha\}}$.
\begin{theorem}\label{thm:truncate}
Let $X$ be a DGMRD random variable with support inbetween $L$ and $H$ with $0\le L<H\le +\infty$. Then, 
\begin{enumerate}[label=$\(\roman*\)\;$, leftmargin=0cm,itemindent=.5cm,labelwidth=\itemindent,labelsep=0cm, align=left, noitemsep, nolistsep]
\item for any $\lambda>0$, the random variable $\lambda X$ is DGMRD (i.e., the DGMRD class is closed under positive scale transformations).
\item for any $\alpha\in\(L,H\)$, the left truncated random variable $X_{\alpha}$ has the same GMRD function as $X$ on $\(\alpha,H\)$. In particular, the DGMRD class is closed under left truncations.
\end{enumerate}
\end{theorem}
\begin{wrapfigure}[15]{r}{0.49\textwidth}
\centering\vspace{-0.2cm}
\includegraphics[width=0.45\textwidth]{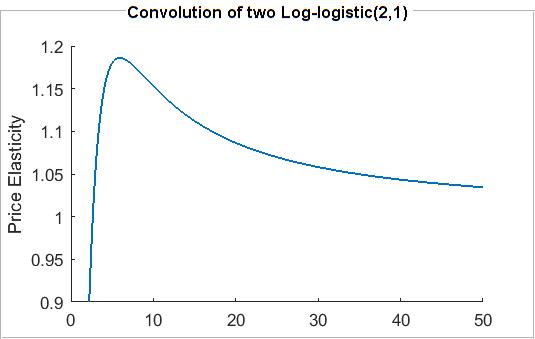}
\caption{\label{loglog} The price elasticity (inverse of the GMRD function) for the convolution of 2 standard Log-logistic$\(k=2\)$ random variables.}
\end{wrapfigure}
In Proposition 1, \cite{Ba13} establish that IGFR distributions are closed under right truncations as well. It remains unclear whether DGMRD distributions are also closed under right truncations or not. On the other hand, as expected, the DGMRD class inherits some closure counterexamples from the IGFR class. \cite{Ba13} illustrate that the IGFR property is not preserved under shifting and convolutions. Both of their examples establish the same conclusions for the DGMRD property, as shown below. \par
Using their notation, the GMRD function of the Pareto distribution of the second kind (Lomax distribution) is $\e\(x\)=\frac{1}{k-1}\(\frac{B-A}{x}+1\)$, for $x\ge A$, where $A$ denotes the location parameter. Hence, when $A=0$ (i.e., no shift) or $A<B$, the GMRD is decreasing, whereas, for $A>B$, the GMRD function is increasing. Similar to the behavior exhibited by the GFR function, the GMRD function is constant for $A=B$, and, in particular for $A=B=1$, which corresponds to the standard Pareto distribution. To show that the IGFR class is not closed under convolution, \cite{Ba13} consider the sum of two log-logistic distributions. The log-logistic distribution is IGFR, and, hence, DGMRD. Using their formula for $F$, one may establish numerically that the price elasticity $\varepsilon\(p\)=\e\(p\)^{-1}$ is first increasing and then decreasing, as can be seen in \Cref{loglog}. 

\section{Discussion: Summary, Limitations \& Directions for Future Work}\label{sec:discussion}
In this paper, we considered the revenue maximization problem of a seller in a market with linear stochastic demand. We expressed the price elasticity of expected demand in terms of the mean residual demand (MRD) function, cf. \eqref{mrl}, of the demand distribution and characterized the seller's optimal prices as fixed points of the MRD function. This led to the description of markets with increasingly elastic demand purely in terms of the demand distribution and in turn, to a novel unimodality condition. Namely, the seller's optimal price in a linear stochastic market exists and is unique if the demand distribution has the decreasing generalized mean residual demand (DGMRD) property and finite second moment. Motivated by the fact that DGMRD distributions strictly generalize the widely used distributions with increasing generalized failure rate (IGFR), we then turned to the study of more technical, yet economically interpretable properties of DGMRD distributions.\par
While these findings expand our understanding on the distributions that are useful in economic applications, the DGMRD class has its own limitations. First, its applicability appears to be limited to the case of linear demand. However, the extensive use of the linear model in economic problems even as a meaningful approximation of general demand curves, cf. \cite{Co15,Ch17} among others, and its detailed analysis via the currently derived technical results, cf. \cite{Leon20,Leo20,Leon21}, provide further support for the study of the DGMRD class of distributions.\footnote{Preliminary versions of these works appear in \cite{Bel18,Kok18}.} In any case, demonstrating the applicability of the DGMRD property beyond the linear setting poses an interesting open direction for future research. \par
Second, while DGMRD distributions strictly generalize IGFR distributions, the latter are already sufficiently inclusive and more easy to handle under the simplifying assumption of the existence of a density. The advantage of analytical tractability becomes apparent when studying the relationship between IFR, IGFR, DMRD and DGMRD distributions. Despite the partial understanding obtained in \Cref{sub:classes} and the illustration in \Cref{fig:relations}, some questions remain difficult to answer precisely due to the technical challenges that arise when handling the integrals that appear in the DGMRD condition. Overcoming this challenges could yield further insight in two directions: (1) In deriving a less restrictive condition under which a DGMRD distribution is also IGFR, cf. \Cref{classes}-(ii) and (2) In exploring whether the DMRD property implies the IGFR property when restricting attention to absolutely continuous random variables with connected support since all provided counterexamples, i.e., examples of DMRD that are not IGFR, violate precisely (one of) these two necessary conditions of the IGFR property, cf. \Cref{sub:classes}. \par
Another manifestation of the technical challenges that stem from handling the more involved DGMRD condition instead of the more tractable IGFR condition appears in the analytical study of the price elasticity curve. Since the latter is fully determined by the MRD function of the demand distribution, cf. equation \eqref{elasticity}, its analysis hinges on the shape of the MRD function. In turn, the MRD function is represented via an indefinite integral, cf. equation \eqref{mrl} and for arbitrary distributions it may behave in many different ways, see \cite{Sh07,Lax06}, which renders the study of the price elasticity curve technically challenging. Yet, based on the current findings, an interesting extension is to study the location of the points of unitary elasticity -- i.e., the seller's optimal prices -- for DMRD, DGMRD or even arbitrary demand distributions. In concrete terms, the main technical challenge to overcome here is to analytically argue about the \emph{location} and the \emph{qualitative properties of the fixed points} of the MRD function. \par
Finally, while the argument of analytical tractability of the IGFR class is arguably important, the GMRD function arises naturally in revenue management applications on pricing (or stocking) decisions under stochastic demand that are not covered by the IGFR condition. Accordingly, both the GMRD function and the DGMRD condition can be of broader interest to the operations research literature, at least in terms of a more succinct and unified representation. In a concrete application, that also serves as a natural motivation to study DGMRD random variabless, mixtures of (possibly IGFR) distributions over disjoint intervals, although important for scenario-analysis of economic models, are not covered by the IGFR property, cf. \Cref{sub:mixtures}. The DGMRD property seems to offer a promising resolution to this deadlock. However, as illustrated in \Cref{ex:uniform}, mixtures of DGMRD distributions may also not be DGMRD. Thus, an important open question this direction is to study closure propertie of the DGMRD class of distributions under mixtures, i.e., to determine necessary and/or sufficient conditions under which mixtures of IGFR or DGMRD distributions are again DGMRD. 
\section*{Acknowledgements}
Stefanos Leonardos gratefully acknowledges support by a scholarship of the Alexander S. Onassis Public Benefit Foundation and partial support from Ministry of Education (MOE) of Singapore and the MOE AcRF Tier 2 Grant 2016-T2-1-170.

\bibliographystyle{plain} 
\bibliography{distbib}
\end{document}